\documentclass[a4paper,12pt]{article}
    
    \usepackage[T2A]{fontenc}
    \usepackage[utf8]{inputenc}
    \usepackage[english]{babel}
    \usepackage[a4paper,hmargin=2.5cm,vmargin=2.5cm]{geometry}
    
    \usepackage{amsmath,amsfonts,amssymb,amsthm,mathtools,dsfont}
    \usepackage{icomma}
    \usepackage{csquotes}  
    \usepackage{etoolbox}
    \usepackage[shortlabels]{enumitem}

    \usepackage{euscript}
    \usepackage{mathrsfs}
    \usepackage{graphicx,xcolor}
    \usepackage[matrix,arrow,curve]{xy}
    
    \usepackage[
    backend=biber,
    style=numeric,
    sorting=ynt
    ]{biblatex}
    \renewbibmacro{in:}{}  
    \usepackage{cmap}  
    \usepackage{hyperref}  

    \numberwithin{equation}{section}  
    
    \let\Im\relax
    \DeclareMathOperator{\Im}{Im}
    \DeclareMathOperator{\Ker}{Ker}
    \DeclareMathOperator{\Coker}{Coker}
    \DeclareMathOperator{\Id}{Id}

    \DeclareMathOperator{\rk}{rk}
    \DeclareMathOperator{\End}{End}
    \DeclareMathOperator{\Hom}{Hom}
    \DeclareMathOperator{\Ext}{Ext}
    
    \DeclareMathOperator{\tr}{tr}

    \DeclareMathOperator{\sEnd}{\mathcal{E}\! \mathit{nd}}
    \DeclareMathOperator{\CHom}{\underline{Hom}}
    \DeclareMathOperator{\Coh}{\mathbf{Coh}}
    \DeclareMathOperator{\mMod}{-\mathbf{Mod}}
    \DeclareMathOperator{\Spec}{Spec}
    \DeclareMathOperator{\Def}{Def}

    \title{On Conormal Lie Algebras of Feigin--Odesskii Poisson Structures}
    \author{Leonid Gorodetsky\thanks{HSE University, Moscow, Russia \quad \tt{ls.gorod.9@gmail.com}}\and Nikita Markarian\thanks{Max Planck Institute for Mathematics, Bonn, Germany \quad \tt{nikita.markarian@gmail.com}} }
   
    \date{}
    
\begin{document} 
    
    \newtheorem{Proposition}{Proposition}[section]
    \newtheorem{Lemma}{Lemma}[section]
    \newtheorem{Corollary}{Corollary}[section]
    \newtheorem{Theorem}{Theorem}
    \newtheorem{Example}{Example}[section]
    \newtheorem{Definition}{Definition}[section]
    \newtheorem*{Acknowledgments}{Acknowledgments}
    
    \newcommand{\mb}{\mathbb}
    \renewcommand{\phi}{\varphi}
    \renewcommand{\O}{\mathcal{O}}
    \newcommand{\D}{\mathcal D}
    \newcommand{\eps}{\varepsilon}
    \newcommand{\vertsum}[2]{\ensuremath{\underset{\displaystyle{#2}}{\overset{\displaystyle{#1}}{\oplus}}}}
    \newcommand{\Lotimes}{\overset{L}{\otimes}}

    \maketitle
    
    \begin{abstract}
        The main result of the paper is a description of conormal Lie algebras of Feigin--Odesskii Poisson structures. In order to obtain it we introduce a new variant of a definition of a Feigin--Odesskii Poisson structure: we define it using a differential on the second page of a certain spectral sequence. In the general case this spectral sequence computes morphisms and higher $\Ext's$ between filtered objects in an abelian category. Moreover, we use our definition to give another proof of the description of symplectic leaves of Feigin--Odesskii Poisson structures.
    \end{abstract}	
    
    \tableofcontents
    
    \section{Introduction}
    
    Let $C$ be an elliptic curve over an algebraically closed field $k$ of characteristic 0 and let $V$ be a stable vector bundle (e.g. a line bundle) of degree $n>0$ on $C$. The Feigin--Odesskii Poisson structure is a remarkable Poisson structure on the projective space $\mb{P}\Ext^1(V, \O_C)$ with the following key property. Consider extensions of the form
    \begin{equation}\label{extension_intro}
        \xymatrix{
            0\ar[r]&\O_C\ar[r]&E\ar[r]&V\ar[r]&0,
        }
    \end{equation}
    which are parameterized by the points of $\Ext^1(V, \O_C)$. One can ask when different extensions lead to the same vector bundle $E$. Firstly note that proportional extensions always lead to the same vector bundle $E$, so it is convenient to consider $\mb{P}\Ext^1(V, \O_C)$ instead of $\Ext^1(V, \O_C)$. Then $\mb{P}\Ext^1(V, \O_C)$ is a disjoint union (as a set) of the isomorphism classes of $E$ (two extensions belong to one isomorphism class of $E$ if they lead to the same vector bundle $E$). The Feigin-Odesskii Poisson structure gives a local description of the isomorphism classes of $E$: connected components of isomorphism classes of $E$ are symplectic leaves of the Feigin--Odesskii Poisson structure on $\mb{P}\Ext^1(V, \O_C)$. 
    
    Feigin--Odesskii Poisson structures were introduced in 1995 by B. Feigin and A. Odesskii in \cite{feigin1995vector}, and they claimed the fact about symplectic leaves, but the proof was not included in the paper. Three years later in \cite{Polishchuk1998} A. Polishchuk defined some Poisson structure on the same space using completely different approach, and only 20 years later in \cite{PolishchukHua2018} Z. Hua and Polishchuk proved that if $V$ is a line bundle then the latter Poisson structure coincides with the Feigin--Odesskii Poisson structure. For this reason Poisson structures defined by Polishchuk are called Feigin--Odesskii Poisson structures as well.
    
    In \cite{PolishchukHua2018} Hua and Polishchuk studied shifted Poisson structures on derived stacks and Feigin--Odesskii Poisson structures appeared as a special case of their construction. In \cite{PolishchukHua2020} they continued studying properties of Feigin--Odesskii Poisson structures and proved that one can recover the elliptic curve $C$ from the Poisson structure. Moreover, they gave another interpretation of Feigin--Odesskii Poisson structures in terms of the triple Massey product (see \cite[Lemma 2.1]{PolishchukHua2020}). A connection between Feigin--Odesskii Poisson structures and secant varieties of $C$ was used in \cite{Polishchuk2022, markarian2024compatible}. In \cite{chirvasitu2023symplectic} secant varieties were used to prove the fact about symplectic leaves in the case when $V$ is a line bundle. In \cite{markarian2024compatible} N. Markarian and Polishchuk studied when different Feigin--Odesskii Poisson structures are compatible, and in this paper they also provided a description of conormal Lie algebras of the Feigin--Odesskii Poisson structure in the case when $V$ is a line bundle.
    
    In the present paper we give a new variant of a definition of Feigin–Odesskii Poisson structures based on a certain spectral sequence and then we use it to give a simple proof of the fact about symplectic leaves (Theorem \ref{thm_symplectic_leaves}) and describe conormal Lie algebras of Feigin--Odesskii Poisson structures (Theorem \ref{thm_conormal_algebras}).
    
    The paper is organized as follows. In Section \ref{section_preliminaries} we provide preliminary facts about stable vector bundles on elliptic curves, the Serre duality on elliptic curves, Poisson structures, and first order deformations of coherent sheaves. 
    
    In Section \ref{section_FO_brackets} we introduce a new variant of a definition of Feigin--Odesskii Poisson structures. In Subsection \ref{subsection_spectral_sequence} we provide a general description of a spectral sequence computing $\Ext's$ between objects of an abelian category that are given as extensions. We describe differentials of this spectral sequence in terms of compositions of $\Ext$'s and triple Massey products. In Subsection \ref{subsection_moduli_space} we show that the space $\mb{P}\Ext^1(V, \O_C)$ of proportionality classes of extensions \eqref{extension_intro} is the moduli space of filtered vector bundles $E\supset L\supset 0$ with fixed associated quotients $E/L\simeq V$ and $L\simeq \O_C$. This description of allows us to view tangent vectors to $P$ as first order deformations of filtered vector bundles. In Subsection \ref{subsection_FO_definition} we apply the construction of the spectral sequence to the extension \eqref{extension_intro} to get a spectral sequence computing $\Ext^\bullet(E, E)$. Then we define the Feigin--Odesskii Poisson structure on $\mb{P}\Ext^1(V, \O_C)$ at the point corresponding to the extension \eqref{extension_intro} as the differential on the second page of the constructed spectral sequence. A description of the differential in terms of the triple Massey product shows that our definition of the Feigin--Odesskii Poisson structure is equivalent to the definition in \cite{PolishchukHua2020}. A similar construction was used in \cite{FKMM_Gmonopoles} to define a symplectic structure coming from a degeneration of the elliptic curve to $\mb{P}^1$. In Subsection \ref{subsection_symplectic_leaves} we prove Theorem \ref{thm_symplectic_leaves} that gives a description of symplectic leaves of Feigin--Odesskii Poisson structures.
    
    In Section \ref{section_conormal_algebras} we describe conormal Lie algebras of Feigin--Odesskii Poisson structures (Theorem \ref{thm_conormal_algebras}), and this is the main result of the paper. In Subsection \ref{subsection_intrinsic_derivative} we give an algebraic definition of the intrinsic derivative of a morphism of vector bundles on an algebraic variety. Then, following \cite{AlanWeinstein}, we formulate the definition of conormal Lie algebras of a Poisson structure based on the intrinsic derivative. Finally, in Subsection \ref{subsection_conormal_FO} we formulate and prove Theorem \ref{thm_conormal_algebras} that gives a description of conormal Lie algebras of the Feigin--Odesskii Poisson structure. In \cite{markarian2024compatible} this result was proved for a special case when $V$ is a line bundle, and in \cite{MarkarianPolishchuk2023compatibleG25} it is used to study compatibility of Feigin--Odesskii brackets in the case if elliptic curves are given as linear sections of the Grassmannian $G(2, 5)$. Our result would be implied by an explicit construction of the symplectic groupoid of the Feigin--Odesskii Poisson structures following \cite{Safronov_ShiftedPoisson}, but our approach is more direct and elementary.

    \begin{Acknowledgments}
        L.G. is very grateful to Alexey Gorodentsev and Mikhail Finkelberg
        and N.M. is very grateful to Alexander Polishchuk for many helpful conversations. N.M.
        would like to thank the Max Planck Institute for Mathematics for hospitality and perfect
        work conditions.
    \end{Acknowledgments}
    
    \section{Preliminaries}\label{section_preliminaries}
    
    Throughout the paper we use the following notations and conventions.
    \begin{itemize}
        \item We work over an algebraically closed field $k$ of characteristic $0$.
        \item By $C$ we denote an elliptic curve over $k$, i.e. a smooth projective curve of genus $1$.
        \item We identify vector bundles on smooth algebraic varieties with locally free coherent sheaves.
        \item For a scheme $X$ over $k$ we denote by $\D(X) = \D^b(\Coh(X))$ the bounded derived category of coherent sheaves on $X$. Similarly, for a $k$-algebra $A$ we denote $\D(A) = \D^b(A\mMod)$
        \item If $F\colon \mathcal{A}\to \mathcal{B}$ is an exact functor between abelian categories, we denote the induced functor between derived categories by the same letter $F\colon \D^b(\mathcal A)\to \D^b(\mathcal B)$.
        \item In order to simplify descriptions of moduli spaces and deformations we use the following convention: we write ''$\simeq$'' for existence of some isomorphism and ''$=$'' for a specified isomorphism.
    \end{itemize}
    
    \subsection{Stable vector bundles on elliptic curves}
    
    Vector bundles on elliptic curves were classified by M. Atiyah in \cite{AtiyahVectorBundles}, and to present the classification it is convenient to use the notion of stable vector bundles. For the following notions and facts we refer to \cite{Stable_Bundles}.
    \begin{Definition}
        Let $V$ be a non-zero vector bundle on a smooth projective curve over $k$.
        \begin{enumerate}
            \item The slope of $V$ is the number
            \begin{equation*}
                \mu(V) = \frac{\deg(V)}{\rk(V)},
            \end{equation*}
            where $\deg(V) = \deg(\det(V))$.
            \item $V$ is called stable if for any non-zero proper coherent subsheaf $V'\varsubsetneq V$
            \begin{equation*}
                \mu(V') < \mu(V).
            \end{equation*}
            Note that $V'$ is automatically a vector bundle since any subsheaf of a locally free coherent sheaf on a smooth projective curve is locally free.
        \end{enumerate}
    \end{Definition}
    
    The following two lemmas show significance of stable vector bundles on elliptic curves.
    \begin{Lemma}
        Let $r > 0$ and $n$ be integers and $L$ be a line bundle of degree $n$ on an elliptic curve $C$.
        \begin{enumerate}
            \item A stable vector bundle $V$ on $C$ of rank $r$ and degree $n$ exists if and only if $\gcd(r, n) = 1$.
            \item If $\gcd(r, n) = 1$, there is a unique stable vector bundle on $C$ of rank $r$ and determinant $L$.
        \end{enumerate}
    \end{Lemma}
    
    \begin{Lemma}
        Let $V$ be a vector bundle of rank $r$ and degree $n$ on an elliptic curve $C$. If $\gcd(r, n) = 1$, then the following properties are equivalent:
        \begin{enumerate}[(i)]
            \item $V$ is stable;
            \item $V$ is simple, i.e. $\End(V) = k$;
            \item $V$ is indecomposable.
        \end{enumerate}
    \end{Lemma}
    Thus, stable vector bundles on an elliptic curve are exactly indecomposable vector bundles with coprime rank and degree. For example, any line bundle is stable. The following lemma will also be useful.
    \begin{Lemma}\label{stableHom}
        Let $V$, $W$ be stable vector bundles on a smooth projective curve such that $\mu(V)>\mu(W)$. Then $\Hom(V, W) = 0$.
    \end{Lemma}
    
    \subsection{The Serre duality on elliptic curves}
    
    Let us fix a trivialization $\omega_C\simeq \O_C$ of the canonical line bundle on the elliptic curve $C$. The Serre duality (see e.g. \cite{Hartshorne_AG}) states that for any vector bundle $V$ on $C$ there is a functorial non-degenerate pairing
    \begin{equation*}\label{Serre_duality_1}
        H^0(C, V)\otimes H^1(C, V^*) \to k,
    \end{equation*}
    and, in particular, this gives a canonical isomorphism
    \begin{equation*}
        H^1(C, \O_C) = k.
    \end{equation*}
    For any vector bundles $V_1, V_2$ on $C$ the Serre duality applied to $V = V_1^*\otimes V_2$ gives a functorial non-degenerate pairing
    \begin{equation}\label{Serre_duality_2}
        \Hom(V_1, V_2)\otimes \Ext^1(V_2, V_1)\to k.
    \end{equation}
    The Serre duality pairing is closely connected with the trace map. For a vector bundle $V$ on $C$ there is a natural trace map
    \begin{equation}\label{trace}
        \tr\colon \sEnd(V)\to \O_C,
    \end{equation}
    and it induces trace maps on cohomology:
    \begin{align}
        \tr\colon \End(V)\to H^0(C, \O_C) = k, \label{trace_end}\\
        \tr\colon \Ext^1(V, V)\to H^1(C, \O_C) = k.\label{trace_ext}
    \end{align}
    Denote the kernels of these three trace maps by $\sEnd(V)_0$, $\End(V)_0$ and $\Ext^1(V, V)_0$, respectively.
    The trace map \eqref{trace} and the section $\O_C\xrightarrow{\Id_V} \sEnd(V)$ give a decomposition
    \begin{equation*}
        \sEnd(V) = \O_C\oplus \sEnd(V)_0
    \end{equation*}
    and the corresponding decomposition of cohomology:
    \begin{align*}
        \End(V) = k \oplus \End(V)_0,\\
        \Ext^1(V, V) = k\oplus \Ext^1(V, V)_0.
    \end{align*}
    The following lemma follows from functoriality of the Serre duality.
    \begin{Lemma}
        Let $C$ be an elliptic curve.
        \begin{enumerate}
            \item For a vector bundle $V$ on $C$ the trace map \eqref{trace_ext}
            \begin{equation*}
                \tr\colon \Ext^1(V, V)\to k
            \end{equation*}
            coincides with the map given by the Serre duality pairing \eqref{Serre_duality_2} with $\Id_V\in \End(V)$.
            \item For two vector bundles $V_1, V_2$ on $C$ the Serre duality pairing 
            \begin{equation*}
                \Hom(V_1, V_2)\otimes \Ext^1(V_2, V_1)\to k
            \end{equation*}
            can be obtained in two other ways: take a composition in any order (there are two variants) and then apply the trace map:
            \begin{equation*}
                \xymatrix{
                    \Hom(V_1, V_2)\otimes\Ext^1(V_2, V_1)\ar[r]\ar[d]&\Ext^1(V_1, V_1)\ar[d]^{\tr}\\
                    \Ext^1(V_2, V_2)\ar[r]^{\;\;\tr}& k.
                }
            \end{equation*}
        \end{enumerate}
    \end{Lemma}
    
    Note that if a vector bundle $V$ is stable then 
    $$\dim \Ext^1(V, V) = \dim \End(V) = 1,$$
    and hence the trace map $\Ext^1(V, V)\to k$ is an isomorphism. Then we get a simple and useful corollary.
    
    \begin{Corollary}\label{Serre_Stable}
        Let $V_1$ and $V_2$ be stable vector bundles on an elliptic curve $C$. Then under identifications $\Ext^1(V_1, V_1) = k$ and $\Ext^1(V_2, V_2) = k$ given by the trace maps, the Serre duality pairing
        \begin{equation*}
            \Hom(V_1, V_2)\otimes \Ext^1(V_2, V_1)\to k
        \end{equation*}
        is identified with the composition in any order.
    \end{Corollary}
    
    \subsection{Poisson structures}
    
    Let $X$ be a smooth algebraic variety over $k$.
    \begin{Definition}\label{definition_Poisson}
        A Poisson structure on $X$ is a $k$-bilinear operation
        \begin{equation}\label{Poisson_bracket}
            \{-,-\}\colon \O_X\times \O_X \to \O_X 
        \end{equation}
        such that for any functions $f, g, h\in \O_X$ the following properties hold:
        \begin{enumerate}
            \item Skew-symmetry:
            $$\{f, g\} = - \{g, f\},$$
            \item Jacobi identity:
            $$\{f, \{g, h\}\} + \{g, \{h, f\}\} + \{h, \{f, g\}\} = 0,$$
            \item Leibniz rule:
            $$\{f, gh\} = \{f, g\}h + g\{f, h\}.$$
        \end{enumerate}
        The pair $(X, \{-, -\})$ is called a smooth Poisson variety.
    \end{Definition}
    Equivalently, one can also view a Poisson structure as a bivector field $\pi\in \Gamma(X, \bigwedge^2(TX))$ such that $[\pi, \pi]=0$, where $[-, -]$ is the Schouten-Nijenhuis bracket. Such bivector field is connected with the bracket from Definition \ref{definition_Poisson} by the formula
    $$\{f, g\} = \pi(df, dg),$$
    and the condition $[\pi, \pi] = 0$ is equivalent to Jacobi identity. Since a bivector field can be viewed as a skew-symmetric map of vector bundles $T^*X\to TX$, a Poisson structure can be also defined as a map of vector bundles
    \begin{equation}\label{Poisson_map}
        \pi\colon T^*X\to TX,
    \end{equation}
    satisfying some additional condition which is equivalent to Jacobi identity. For any point $x\in X$ the Poisson structure \eqref{Poisson_map} induces a map
    $$\pi_x\colon T_x^*X\to T_xX,$$
    and the rank of this map is called the rank of the Poisson structure $\pi$ at the point $x$. Since $\pi_x$ is skew-symmetric, its rank is always even.
    
    \begin{Definition}
        Let $X$ be a smooth algebraic variety and $\pi\colon T^*X\to TX$ a  Poisson structure on $X$. An algebraic symplectic leaf is a maximal connected locally closed subset $Z\subset X$ such that for any point $x\in Z$ the tangent space at $x$ to $Z$ coincides with the image of the Poisson structure:
        \begin{equation}
            T_xZ = \Im(\pi_x)\subset T_xX.
        \end{equation}
    \end{Definition}
    
    In the case $k=\mb{C}$ Frobenius theorem implies that each point of $X$ is contained in an analytic symplectic leaf, which is not necessarily algebraic. Thus, even existence of algebraic symplectic leaves is not guaranteed.
    \begin{Proposition}\label{Proposition_leaf_smooth}
        Let $X$ be a smooth algebraic variety and $\pi\colon T^*X\to TX$ be a Poisson structure on $X$. If $Z\subset X$ is an algebraic symplectic leaf of $\pi$, then $Z$ is a smooth algebraic variety.
    \end{Proposition}
    \begin{proof}
        Since $\dim T_x Z$ is upper semicontinuous for $x\in Z$, $\rk \pi_x$ is lower semicontinuous for $x\in Z$, and $\dim T_x Z = \rk \pi_x$ for $x\in Z$, it follows that $\dim T_x Z$ is constant on $Z$. Thus, $Z$ is smooth.
    \end{proof}
    Throughout the paper by symplectic leaves we always mean algebraic symplectic leaves.
    
    \subsection{First order deformations of coherent sheaves}\label{subsection_deformations}
    
    Denote by $k[\eps]$ the ring of dual numbers $k[t]/t^2$ and denote its spectrum by $D = \Spec k[\eps]$. Let X be an algebraic variety over $k$. Consider two Cartesian squares
    $$
    \xymatrix{
        X\times D\ar[r]^(0.55)q\ar[d]_p & D\ar[d]^{p'} && X\times D\ar[r]^(0.55){q} & D\\
        X\ar[r]^{q'} & pt, && X\ar[r]^{q'}\ar[u]^i & pt,\ar[u]_{i'}
    }
    $$
    where
    \begin{itemize}
        \item $pt = \Spec k$,
        \item $q$ is the natural projection $X\times D\to D$,
        \item $p'$ and $p$ are induced by the natural embedding $k\to k[\eps]$,
        \item $i'$ and $i$ are induced by the map $k[\eps]\to k$, $a + b\eps\mapsto a$.
    \end{itemize}
    Note that morphisms $p', p, i', i$ are affine, morphisms $p', p, q', q$ are flat, and $p\circ i = id_X$.
    
    \begin{Definition}
        Let $E$ be a coherent sheaf on $X$. A first order deformation of $E$ is a coherent sheaf $\tilde E$ on a scheme $X\times D$ such that
        \begin{enumerate}
            \item $\tilde E$ is flat over $D$, i.e. the functor $\tilde E\otimes q^*(-)\colon \Coh(D)\to \Coh(X\times D)$ is exact,
            \item $i^*\tilde E = E$, i.e. the restriction of $\tilde E$ to X is identified with $E$.
        \end{enumerate}
        Two first order deformations $\tilde E_1$ and $\tilde E_2$ of $E$ are equivalent if there is an isomorphism $\tilde E_1\simeq \tilde E_2$ that agrees with identifications $i^*\tilde E_1 = E$ and $i^*\tilde E_2 = E$. The set of equivalence classes of first order deformations is denoted by $\Def(E)$.
    \end{Definition}
    Since we do not consider higher order deformations in this paper, we will usually omit the term "first order". Fix a coherent sheaf $E$ on $X$. There is a well-known bijection between $\Def(E)$ and $\Ext^1(E, E)$ such that zero element in $\Ext^1(E, E)$ corresponds to the trivial deformation $\tilde E = p^*E$ (see e.g. \cite{Hartshorne_DeformationTheory}). Let us also describe this bijection $\Def(E)\to \Ext^1(E, E)$ using derived categories.
    
    Firstly consider the $k[\eps]$-module $k$ where $\eps$ acts trivially and consider the element $\xi\in \Ext_{k[\eps]}^1(k, k)$ corresponding to the extension
    \begin{equation}\label{k_extension}
        \xymatrix @R=.5pc{
            0\ar[r]&k\ar[r]^(0.45)\eps&k[\eps]\ar[r]&k\ar[r]&0.
        }
    \end{equation}
    Starting from a deformation $\tilde E$ of $E$ we are going to construct an element $T\in \Ext^1(E, E)$. Consider the functor 
    \begin{equation}\label{deformation_functor}
        p_*(\tilde E\otimes q^*(-))\colon \Coh(D)\to \Coh(X).
    \end{equation}
    \begin{Lemma}\label{deformation_functor_lemma}
        One has
        \begin{enumerate}
            \item $q^*k = i_*\O_X$,
            \item $\tilde E\otimes q^*k = i_*E$,
            \item $p_*(\tilde E\otimes q^* k) = E$.
        \end{enumerate}
    \end{Lemma}
    \begin{proof}
        Since the $k[\eps]$-module $k$ is the pushforward $i'_* \O_{pt}$ of the structure sheaf $\O_{pt}$ of the point, the flat base change implies
        \begin{equation*}
            q^* k = q^*i'_*\O_{pt} = i_*q'^*\O_{pt} = i_*\O_X.
        \end{equation*}
        Then by the projection formula
        \begin{equation*}
            \tilde E\otimes q^* k = \tilde E\otimes i_*\O_X = i_*E.
        \end{equation*}
        Finally, $p_*(\tilde E\otimes q^* k) = p_*i_*E = E.$
    \end{proof}
    The functor \eqref{deformation_functor} is exact since $\tilde E$ is flat over $D$, so it naturally extends to the functor
    \begin{equation}\label{deformation_functor_2}
        p_*(\tilde E\otimes q^*(-))\colon \D(k[\eps])\to \D(X).
    \end{equation}
    Applying this functor to the morphism $\xi\colon k\to k[1]$ in $\D(k[\eps])$ and using Lemma \ref{deformation_functor_lemma} we get a morphism 
    $$p_*(\tilde E\otimes q^* \xi)\colon E\to E[1]$$
    in $\D(X)$, i.e. an element of $\Ext^1(E, E)$. In other words, the deformation $\tilde E$ corresponds to the element
    \begin{equation}\label{deformation_element}
        T = p_*(\tilde E\otimes q^*\xi)\in \Ext^1(E, E).
    \end{equation}
    corresponding to the extension
    $$
    \xymatrix @R=.5pc{
        0\ar[r]&E\ar[r]&p_*\tilde E\ar[r]&E\ar[r]&0,
    }$$
    obtained as a tensor product over $k[\eps]$ of $\tilde E$ and the extension \eqref{k_extension}.
    \begin{Lemma}\label{lemma_deformation}
        If $E$ is a locally free sheaf on an algebraic variety $X$ then any deformation $\tilde E$ of $E$ is a locally free sheaf on $X\times D$.
    \end{Lemma}
    \begin{proof}
        For any affine open subset $U\subset X$ there are no non-trivial deformations of $E|_U$ since $\Ext^1_U(E|_U, E|_U) = 0$. Hence, $\tilde E|_{U\times D} = (p|_{U\times D})^*(E|_U)$ and then $\tilde E$ is a locally trivial sheaf. 
    \end{proof}
    
    For locally free sheaves on $X$ there is a convenient description of deformations in terms of Čech cocycles. Let $E$ be a locally free sheaf on $X$. Choose a covering $\{U_i\}_{i\in I}$ of $X$ by affine open subsets $U_i$, choose trivializations of $E$ on $U_i$, and denote the corresponding transition functions by $g_{ij}$ for $i, j\in I$. Let $\tilde E$ be a deformation of $E$. By Lemma \ref{lemma_deformation}, $\tilde E$ can be trivialized on affine open subsets $U_i\times D$ of $X\times D$, and the transition functions $\tilde g_{ij}$ of $\tilde E$ can be written as $\tilde g_{ij} = g_{ij} + \eps \cdot h_{ij}$ for some matrices $h_{ij}$ of functions on $X$. One can check that $h_{ij}$ is a $1$-cocycle on $X$ representing an element of $H^1(X, \sEnd(E)) = \Ext^1(E, E)$ and, in fact, this construction gives the same bijection between deformations of $E$ and $\Ext^1(E, E)$ as we described earlier.
    
    \begin{Lemma}\label{lemma_deformation_tensor}
        \begin{enumerate}
            \item Let $E_1$, $E_2$ be vector bundles on an algebraic variety $X$ and let $\tilde E_1$, $\tilde E_2$ be deformations of $E_1, E_2$ corresponding to elements $T_1\in \Ext^1(E_1, E_1)$, $T_2\in \Ext^1(E_2, E_2)$, respectively. Then $\tilde E_1\otimes \tilde E_2$ is the deformation of $E_1\otimes E_2$ corresponding to $T_1\otimes \Id_{E_2} + \Id_{E_1}\otimes T_2\in \Ext^1(E_1\otimes E_2, E_1\otimes E_2)$.
            \item Let $E$ be a vector bundle on an algebraic variety $X$ and let $\tilde E$ be a deformation of $E$ corresponding to an element $T\in \Ext^1(E, E)$. Then the dual sheaf $\tilde E^*$ on $X\times D$ is the deformation of $E^*$ corresponding to $-T\in \Ext^1(E^*, E^*)=\Ext^1(E, E)$.
            \item Let $E$ be a vector bundle on an algebraic variety $X$ and let $\tilde E$ be a deformation of $E$ corresponding to an element $T\in \Ext^1(E, E)$. Then $\sEnd(\tilde E)$ is the deformation of $\sEnd(E)$ corresponding to $[-, T]\in \Ext^1(\sEnd(E), \sEnd(E))$.
        \end{enumerate}
    \end{Lemma}
    \begin{proof}
        We will prove only the first statement, other statements are similar. Take an affine open covering $\{U_i\}$ of $X$ and let $g^1_{ij}$ and $g^2_{ij}$ be transition functions of $E_1$ and $E_2$, respectively. Then transition functions of deformations $\tilde E_1$ and $\tilde E_2$ can be written as $g^1_{ij}+\eps h^1_{ij}$ and $g^2_{ij}+\eps h^2_{ij}$, where Čech cocycles $h^1_{ij}$ and $h^2_{ij}$ represent elements $T_1$ and $T_2$. Then the tensor product $\tilde E_1\otimes \tilde E_2$ has transition functions
        $$(g^1_{ij}+\eps h^1_{ij})\otimes (g^2_{ij}+\eps h^2_{ij}) = g^1_{ij}\otimes g^2_{ij} + \eps (h^1_{ij}\otimes g^2_{ij} + g^1_{ij}\otimes h^2_{ij}),$$
        and the cocycle $h^1_{ij}\otimes g^2_{ij} + g^1_{ij}\otimes h^2_{ij}$ represents the element $T_1\otimes \Id_{E_2}+\Id_{E_1}\otimes T_2$.
    \end{proof}
    
    \section{Feigin--Odesskii Poisson Structures}\label{section_FO_brackets}
    
    \subsection{The spectral sequence}\label{subsection_spectral_sequence}
    
    Our definition of Feigin--Odesskii Poisson structures is based on the following general spectral sequence.
    
    Let $R$ be a commutative ring and $\mathcal A$ be an $R$-linear abelian category with enough injective objects. Consider two extensions
    \begin{align}
        &\xymatrix{
            0\ar[r]&A'\ar[r]^{i_A}&A\ar[r]^{\pi_A}&A''\ar[r]&0,
        }\label{extension1}\\
        &\xymatrix{
            0\ar[r]&B'\ar[r]^{i_B}&B\ar[r]^{\pi_B}&B''\ar[r]&0
        }\label{extension2}
    \end{align}
    in $\mathcal A$ and the corresponding elements $\phi \in \Ext^1(A'', A')$, $\psi\in \Ext^1(B'', B')$. We start with the following question: how to describe $\Ext^\bullet(A, B)$ in terms of $\Ext^\bullet(A', B')$, $\Ext^\bullet(A', B'')$, $\Ext^\bullet(A'', B')$, $\Ext^\bullet(A'', B'')$. For example, if extensions $\phi$ and $\psi$ are trivial 
    then we can view elements of $\Ext^k(A, B)$ as $2\times 2$-matrices
    \begin{equation}\label{2matrices}
        \begin{pmatrix}
            u^k& t^k\\
            s^k& v^k
        \end{pmatrix},\;\text{ where }
        \begin{matrix}
            &u^k\in \Ext^k(A', B'), &\quad t^k\in \Ext^k(A'', B'),\\
            &s^k\in \Ext^k(A', B''), &\quad v^k\in \Ext^k(A'', B'').
        \end{matrix}
    \end{equation}
    In the general case there is a natural spectral sequence computing $\Ext^\bullet(A, B)$ with such $2\times 2$-matrices on the first page. The spectral sequence can be constructed as follows. Short exact sequences \eqref{extension1}, \eqref{extension2} give two term filtrations on $A$ and $B$, and one can use horseshoe lemma to construct injective resolutions $I_A^\bullet$, $I_B^\bullet$ of $A$ and $B$ with corresponding filtations on them. Then $\Ext^\bullet(A, B)$ can be computed as cohomology of the $\Hom$-complex $\CHom^\bullet(I_A^\bullet, I_B^\bullet)$ and filtrations on $I_A^\bullet$ and $I_B^\bullet$ induce a filtration on this complex, so we get a spectral sequence computing $\Ext^\bullet(A, B)$. A straightforward computation gives an explicit description of this spectral sequence. We provide only the final answer.
    
    The first page $E_1^{\bullet, \bullet}$ of the spectral sequence looks as follows.
    $$
    \xymatrix{
        \ar@{-}@<4ex>[ddddddd]  &      \hphantom{---------}       &    \hphantom{---------}    &  \hphantom{---------}   &         \\
        \vdots  &      \vdots       &    \vdots    &  \vdots   &           \\
        3  &  \Ext^2(A',B'')\ar[r]^{(-1)\cdot \psi\circ -}_{-\circ\phi}  & \vertsum{\Ext^3(A', B')}{\Ext^3(A'',B'')} \ar[r]^{-\circ\phi}_{\psi\circ -}& \Ext^4(A'', B')  &\\
        2  &  \Ext^1(A', B'')\ar[r]^{\psi\circ -}_{-\circ\phi} & \vertsum{\Ext^2(A', B')}{\Ext^2(A'',B'')}\ar[r]^{-\circ\phi}_{(-1)\cdot \psi\circ -} &    \Ext^3(A'',B')      &        \\
        1  &  \Hom(A',B'')\ar[r]^{(-1)\cdot \psi\circ -}_{-\circ\phi}  & \vertsum{\Ext^1(A', B')}{\Ext^1(A'',B'')} \ar[r]^{-\circ\phi}_{\psi\circ -}& \Ext^2(A'', B')  &\\
        0  &       & \vertsum{\Hom(A', B')}{\Hom(A'',B'')}\ar[r]^{-\circ\phi}_{(-1)\cdot \psi\circ -} & \Ext^1(A'', B')  &\\
        -1  &       &  & \Hom(A'', B')  &\\
        \ar@{-}@<4ex>[rrrr]  & -1             & 0      & 1&
    }
    $$
    This spectral sequence degenerates on the third page, so $E_3^{p, q}=E_\infty^{p, q}$ is isomorphic to the $p$-th associated quotient of $\Ext^{p+q}(A, B)$. Note that diagonals on the first page indeed consist of $2\times 2$-matrices \eqref{2matrices}.
    
    On the second page for $k\ge 0$ the differential $d_2\colon E_2^{-1, k + 1}\to E_2^{1, k}$ or, more explicitly,
    \begin{align*}
        &d_2\colon \{s\in \Ext^k(A', B'')\mid \psi\circ s = 0, \;\; s\circ \phi = 0\} \to \frac{\Ext^{k+1}(A'', B')}{\Ext^k(A', B')\circ\phi + \psi\circ \Ext^k(A'', B'')}
    \end{align*}
    can be expressed using the triple Massey product $MP$: (see \cite{GelfandManin, PolishchukHua2020})
    \begin{equation}
        d_2(s) = MP(\phi, s, \psi).
    \end{equation}
    Let us also provide two explicit constructions of $d_2$, which follow from the definition of Massey product. For $s\in \Ext^k(A', B'')$ such that $s\circ \phi = 0$ and $\psi\circ s = 0$ consider the following diagram with two distinguished triangles in the derived category $\D^b(\mathcal A)$:
    $$
    \xymatrix{
        &A\ar[ld]_{\pi_A}&&&&B\ar[ld]_{\pi_B}&\\
        A''\ar[rr]^\phi_{[1]}&& A'\ar[rr]^s_{[k]}\ar[lu]_{i_A} && B''\ar[rr]^\psi_{[1]}&& B' \ar[lu]_{i_B}.
    }
    $$
    
    The first construction is the following. Since $\psi\circ s = 0$, there is a lift $s'\in \Ext^k(A', B)$ such that $s = \pi_B\circ s'$. Consider the composition $s'\circ \phi$ and note that
    $$\pi_B\circ (s'\circ \phi) = s\circ \phi = 0.$$
    Hence, there is a lift $s''\in \Ext^{k+1}(A'', B')$ such that $i_B\circ s'' = s'\circ \phi$, and then $d_2(s)$ is the class of $s''\in \Ext^{k+1}(A'', B')$. 
    
    The second construction is very similar. Since $s\circ \phi = 0$, there is a lift $s'\in \Ext^k(A, B'')$ such that $s'\circ i_A = s$. Consider the composition $\psi \circ s'$ and note that
    $$(\psi\circ s')\circ i_A = \psi\circ s = 0.$$
    Hence, there is a lift $s''\in \Ext^{k+1}(A'', B')$ such that $s''\circ\pi_A = \psi\circ s'$, and then $d_2(s)$ is the class of $-s''\in \Ext^{k+1}(A'', B')$.
    
    \subsection{The moduli space of filtered vector bundles}
    \label{subsection_moduli_space}
    
    Let $V$ be a stable vector bundle of rank $r >0$ and degree $n>0$ on an elliptic curve $C$. Consider extensions of $V$ by the trivial line bundle $\O_C$, i.e. short exact sequences
    \begin{equation}\label{extension}
        \xymatrix{
            0\ar[r]&\O_C\ar[r]^i&E\ar[r]^\pi&V\ar[r]&0.
        }
    \end{equation}
    As in \cite{feigin1995vector}, we consider the following moduli space of filtered vector bundles.
    \begin{Proposition}
        The projective space $P = \mb{P}\Ext^1(V, \O_C)$ is the moduli space of filtered vector bundles $E\supset L\supset 0$ with fixed associated quotients
        \begin{equation}\label{moduli_iso}
            E/L \simeq V,\quad L\simeq \O_C.
        \end{equation}
        Isomorphisms \eqref{moduli_iso} are not specified, i.e. they are not a part of the data of a filtered vector bundle. Moreover, we throw away the trivial filtered vector bundle $V\oplus \O_C\supset \O_C\supset 0$.
    \end{Proposition}
    \begin{proof}
        Take a point $\langle \phi\rangle\in P$, where a non-zero element $\phi\in \Ext^1(V, \O_C)$ corresponds to an extension of the form \eqref{extension}. We set $L$ to be the image of $\O_C$ in $E$ and $E\supset L\supset 0$ to be the corresponding filtrarion. Note that if we replace $\phi$ by $\lambda\cdot \phi$ for some non-zero $\lambda\in k$ then the corresponding extension will be 
        $$\xymatrix{
            0\ar[r]&\O_C\ar[r]^{\lambda\cdot i}&E\ar[r]^\pi&V\ar[r]&0,
        }
        $$
        and the image of $\O_C$ in $E$ will not change. Hence, the filtered vector bundle $E\supset L\supset 0$ depends only on the proportionality class of $\phi$.
        
        Conversely, consider a filtered vector bundle $E\supset L\supset 0$ and fix some isomorphisms 
        $E/L\simeq V$ and $L\simeq \O_C$. Using these isomorphisms we obtain an extension of the form \eqref{extension} corresponding to some non-zero element $\phi\in \Ext^1(V, \O_C)$. Since $V$ and $\O_C$ are a stable vector bundles, all the isomorphisms $E/L\simeq V$ and $L\simeq \O_C$ differ by multiplication by non-zero constants. Hence, for different choices of isomorphisms we obtain extensions of the form
        \begin{equation}\label{extensions_proportional}
            \xymatrix{
                0\ar[r]&\O_C\ar[r]^{\lambda\cdot i}&E\ar[r]^{\mu\cdot \pi}&V\ar[r]&0
            }
        \end{equation}
        for non-zero $\lambda, \mu\in k$. Since elements of $\Ext^1(V, \O_C)$ corresponding to extensions \eqref{extensions_proportional} are proportional to $\phi$, the proportionality class of $\phi$ does not depend on the choice of isomorphisms.
        
        It is clear that described constructions are inverse to each other. Existence of the universal family on $P=\mb{P}\Ext^1(V, \O_C)$ is shown in \cite{UniversalFamiliesofExtensions}.  
    \end{proof}
    The following proposition shows that $P$ is the $(n-1)$-dimensional projective space.
    \begin{Proposition}\label{proposition_dimensions}
        Let $V$ be a stable vector bundle of degree $n > 0$ on an elliptic curve $C$. Then
        \begin{enumerate}
            \item $\Hom(V, \O_C) = \Ext^1(\O_C, V) = 0,$
            \item $\dim \Ext^1(V, \O_C) =\dim \Hom(\O_C, V) = n.$
        \end{enumerate}
    \end{Proposition}
    \begin{proof}
        Since $\mu(V) > 0 = \mu(\O_C)$ and vector bundles $V$ and $\O_C$ are stable, it follows from Lemma \ref{stableHom} that $\Hom(V, \O_C) = 0$ and then, by the Serre duality, $\Ext^1(\O_C, V) = 0$. For the second part we apply the Riemann-Roch theorem:
        \begin{align*}
            \dim \Hom(\O_C, V) - \dim \Ext^1(\O_C, V) = n.
        \end{align*}
        Hence, $\dim \Hom(\O_C, V) = n$ and, by the Serre duality, $\dim \Ext^1(V, \O_C) = n$.
    \end{proof}
    
    Now let us describe tangent spaces to $P$. On the one hand, $P = \mb{P}\Ext^1(V, \O_C)$ is a projective space, so for a non-zero element $\phi\in \Ext^1(V, \O_C)$ there is an identification
    $$T_\phi P = \Ext^1(V, \O_C)/\langle \phi\rangle,$$
    of degree 1 in $\phi$.
    
    On the other hand, we can describe tangent spaces to $P$ in terms of first order deformations. Recall from Section \ref{subsection_deformations} that we write $D = \Spec k[\eps]$ for the spectrum of dual numbers, and we have natural morphisms $p\colon C\times D\to C$ and $i\colon C\to C\times D$ such that $p\circ i = id_X$. The universal property of $P$ leads to the following definition.
    \begin{Definition}\label{definition_filtered_deformation}
        Let $E\supset L\supset 0$ be a filtered vector bundle on $C$ such that $E/L\simeq V$ and $L\simeq \O_C$. A (first order) deformation of $E\supset L\supset 0$ is a filtered coherent sheaf $\tilde E\supset \tilde L\supset 0$ on $C\times D$ such that
        \begin{enumerate}
            \item $i^*\tilde E = E$ and $i^*\tilde L = L$,
            \item $\tilde E / \tilde L\simeq p^*V$ and $\tilde L\simeq p^*\O_C$, i.e. associated quotients deform trivially.
        \end{enumerate}
        The set of equivalence classes of deformations of $E\supset L\supset 0$ is denoted by $\Def(E\supset L\supset 0)$.
    \end{Definition}
    Thus, if $E\supset L\supset 0$ is a filtered vector bundle corresponding to a point $\langle\phi\rangle\in P$ then the tangent space $T_\phi P$ to $P$ consists of the deformations of $E\supset L\supset 0$:
    $$T_\phi P = \Def(E\supset L\supset 0).$$
    Note that if $\tilde E\supset \tilde L\supset 0$ is a deformation of $E\supset L\supset 0$ then $\tilde L$ and $\tilde E$ are automatically flat over $D$. Hence, in this case $\tilde E$ is a deformation of $E$, and we get a map
    \begin{equation*}
        \Def(E\supset L\supset 0)\to \Def(E)
    \end{equation*}
    that forgets about $\tilde L$.
    
    \subsection{Definition of Feigin--Odesskii Poisson structures}
    \label{subsection_FO_definition}
    
    We start from a stable vector bundle $V$ of degree $n > 0$ on an elliptic curve $C$ and consider proportionality classes of extensions
    $$
    \xymatrix{
        0\ar[r] &\O_C\ar[r] &E\ar[r] &V\ar[r] &0,
    }
    $$
    which are parameterized by $P=\mb{P}\Ext^1(V, \O_C)$. We will define the Feigin--Odeskii Poisson structure $\pi$ on $P$ as a morphism of vector bundles
    $$
    \pi\colon T^*P\to TP.
    $$
    Since $P$ is a projective space, for any non-zero $\phi\in \Ext^1(V, \O_C)$ we have the identification
    $$T_\phi P = \Ext^1(V, \O_C) / \langle \phi \rangle,$$
    of degree 1 in $\phi$, and using the Serre duality pairing
    $$\langle-, -\rangle\colon \Hom(\O_C, V)\otimes \Ext^1(V, \O_C)\to k$$
    we can write
    $$T^*_\phi P = \langle \phi\rangle^\perp \subset \Hom(\O_C, V).$$
    Fix a non-zero element $\phi\in \Ext^1(V, \O_C)$ corresponding to an extension
    \begin{equation}\label{extension_2}
        \xymatrix{
            0\ar[r] &\O_C\ar[r] &E\ar[r] &V\ar[r] &0
        }
    \end{equation}
    and consider the corresponding filtered vector bundle $E\supset L\supset 0$. We can apply to the extension \eqref{extension_2} the construction of the spectral sequence computing $\Ext^\bullet(E, E)$ that was described in Section \ref{subsection_spectral_sequence}.
    
    Since $\Hom(V, \O_C) = 0$, $\Ext^1(\O_C, V)=0$ by Proposition \ref{proposition_dimensions} and all higher $\Ext$'s are zero, the first page $E_1^{\bullet, \bullet}$ of the spectral sequence looks as follows.
    
    \begin{equation}\label{FO_spectral_sequence_nonreduced}
        \xymatrix{
            \ar@{-}@<4ex>[ddd]  &      \hphantom{--------}       &    \hphantom{--------}    &  \hphantom{--------}   &         \\
            1  &  \Hom(\O_C, V)\ar[r]^{(-1)\cdot \phi\circ -}_{-\circ\phi}  & \vertsum{\Ext^1(\O_C, \O_C)}{\Ext^1(V, V)}&   &\\
            0  &       & \vertsum{\Hom(\O_C, \O_C)}{\Hom(V, V)}\ar[r]^{-\circ\phi}_{(-1)\cdot \phi\circ -} & \Ext^1(V, \O_C)  &\\
            \ar@{-}@<4ex>[rrrr]  & -1             & 0      & 1&
        }
    \end{equation}
    The diagonal $p + q = 0$ corresponds to $\End(E)$ and the diagonal $p + q = 1$ corresponds to $\Ext^1(E, E)$. 
    
    There is another way to construct the same spectral sequence. For a filtered vector bundle $E\supset L\supset 0$ consider the sheaf $\sEnd(E)$ with the induced three term filtration on it, and this filtration gives a spectral sequence computing $H^\bullet(C, \sEnd(E)) = \Ext^\bullet(E, E)$. Moreover, one can use the sheaf $\sEnd(E)_0$ instead of $\sEnd(E)$ to get the reduced spectral sequence computing $\Ext^\bullet(E, E)_0$. Since $\sEnd(E) = \sEnd(E)_0\oplus \O_C$, the reduced spectral sequence is the traceless part of the initial spectral sequence \eqref{FO_spectral_sequence_nonreduced} and, moreover, it is a direct summand in it.
    
    Using Corollary \ref{Serre_Stable} we can write the first page of the reduced spectral sequence as follows.
    \begin{equation}\label{FO_definition_spectral_sequence}
        \xymatrix{
            \ar@{-}@<4ex>[ddd]  &      \hphantom{----}       &    \hphantom{----}    &  \hphantom{----}   &         \\
            1  &  \Hom(\O_C, V)\ar[r]^(0.65){\langle \phi, -\rangle}  & k&   &\\
            0  &       & k\ar[r]^(0.35){\cdot \phi} & \Ext^1(V, \O_C)  &\\
            \ar@{-}@<4ex>[rrrr]  & -1             & 0      & 1&
        }
    \end{equation}
    The only non-trivial differential on the second page is a map
    $$d_2\colon \langle \phi\rangle^\perp\to \Ext^1(V, \O_C)/\langle \phi\rangle,$$
    and this map is the same for both spectral sequences \eqref{FO_spectral_sequence_nonreduced} and \eqref{FO_definition_spectral_sequence}.
    \begin{Definition}
        The Feigin--Odesskii Poisson structure $\pi$ on $P = \mb{P}\Ext^1(V, \O_C)$ at a point $\langle\phi\rangle$ is defined as the differential $d_2$ on the second page of the constructed spectral sequence \eqref{FO_definition_spectral_sequence}:
        $$\pi_\phi\colon T^*_\phi P = \langle\phi\rangle^\perp\xrightarrow{d_2} \Ext^1(V, \O_C)/\langle\phi\rangle = T_\phi P.$$
    \end{Definition}
    An explicit description of the differential $d_2$ provided in Section \ref{subsection_spectral_sequence} leads to the following proposition.
    \begin{Proposition}\label{FO_Massey_product}
        Let $V$ be a stable vector bundle of positive degree on $C$ and let $\phi\in \Ext^1(V, \O_C)$ be a non-zero element. Then the Feigin--Odesskii Poisson structure $\pi$ at the point $\langle\phi\rangle\in P$ can be expressed as
        $$\pi_\phi = MP(\phi, -, \phi),$$
        where $MP$ is the triple Massey product. Thus, our definition of the Feigin--Odesskii Poisson structure coincides with the definition in \cite[Lemma 2.1]{PolishchukHua2020}.
    \end{Proposition}
    Note that it follows from Proposition \ref{FO_Massey_product} that $\pi_\phi$ has degree 2 in $\phi$, and hence $\pi$ is a well-defined map of vector bundles $T^*P\to TP$ on $P$.
    
    Since the reduced spectral sequence converges to $\Ext^\bullet(E, E)_0$, it follows immediately that
    \begin{align*}
        &\Ker(\pi_\phi) = \End(E)_0,\\
        &\Coker(\pi_\phi) = \Ext^1(E, E)_0.
    \end{align*}
    As a corollary, we get a formula for rank from \cite{PolishchukHua2020}.
    \begin{Corollary}
        Let $V$ be a stable vector bundle of degree $n>0$ on an elliptic curve $C$. Let $\langle\phi\rangle\in P$ be the point corresponding to a filtered vector bundle $E\supset L\supset 0$. Then the rank of the Feigin--Odesskii Poisson structure $\pi$ at this point equals to
        $$\rk \pi_\phi = n - \dim \End(E).$$
    \end{Corollary}
    \begin{proof}
        Since $\dim T^*_\phi P = n - 1$,
        $$\rk\pi_\phi = n - 1 - \dim \Ker(\pi_\phi) = n - 1 - (\dim \End(E) - 1) = n - \dim\End(E).$$
    \end{proof}
    
   Our construction of the Feigin--Odesskii Poisson structure is functorial in the following sense. Let $A$ be a commutative finitely generated $k$-algebra. Consider the Cartesian square
    $$
    \xymatrix{
        C\times \Spec A\ar[r]^(0.55)q\ar[d]_p & \Spec A\ar[d]^{p'}\\
        C\ar[r]^{q'} & pt
    }
    $$
    and note that the flat base change implies that for locally free sheaves $E, F$ on $C$
    \begin{equation}\label{Ext_A}
        \Ext^\bullet(p^*E, p^*F) = \Ext^\bullet(E, F)\otimes_k A.
    \end{equation}
    as $A$-modules. For a morphism $t\colon \Spec A\to P$ the universal property of $P$ gives a filtered coherent sheaf $\tilde E\supset \tilde L\supset 0$ on $C\times \Spec A$ with $\tilde E/\tilde L\simeq p^*V$ and $\tilde L\simeq p^*\O_C$. Reproducing the definition of the Feigin-Odesskii Poisson structure but for the $A$-linear case, we get a spectral sequence of $A$-modules that computes $\Ext^\bullet(\tilde E, \tilde E)_0$. It follows from \eqref{Ext_A} that the first page looks as follows.
    $$
    \xymatrix{
        \ar@{-}@<4ex>[ddd]  &      \hphantom{----}       &    \hphantom{----}    &  \hphantom{----}   &         \\
        1  &  \Hom(\O_C, V)\otimes_k A\ar[r]  & A&   &\\
        0  &       & A\ar[r] & \Ext^1(V, \O_C)\otimes_k A  &\\
        \ar@{-}@<4ex>[rrrr]  & -1             & 0      & 1&
    }
    $$
    The nontrivial differential $d_2$ on the second page gives a morphism
    $$t^*(T^*P)\to t^*(TP)$$
    of sheaves on $\Spec A$. Globalizing, for any morphism $t\colon U\to P$ for any scheme $U$ we get the functorial map $t^*(T^*P)\to t^*(TP)$, which can be viewed as a functorial version of the Feigin--Odesskii Poisson structure.
    
    \subsection{Symplectic leaves of Feigin--Odesskii Poisson structures}
    \label{subsection_symplectic_leaves}
    
    As before, fix a stable vector bundle $V$ of degree $n>0$ on an elliptic curve $C$ and consider the moduli space $P=\mb{P}\Ext^1(V, \O_C)$ of filtered vector bundles $E\supset L\supset 0$ with associated quotients $E/L \simeq V$, $L\simeq \O_C$. Then $P$ is a disjoint union (as a set) of isomorphism classes of $E$: we say that two filtered vector bundles $E_1\supset L_1\supset 0$ and $E_2\supset L_2\supset 0$ are in the same isomorphism class of $E$ if vector bundles $E_1$ and $E_2$ are isomorphic. The following essential property of Feigin--Odesskii Poisson structures was initially stated in \cite{feigin1995vector}. We provide a simple proof for it.
    \begin{Theorem}\label{thm_symplectic_leaves}
        Connected components of isomorphism classes of $E$ are symplectic leaves of the Feigin--Odesskii Poisson structure $\pi$ on $P$. 
    \end{Theorem}
    \begin{proof}
        Take a point $\langle\phi\rangle\in P$ corresponding to a filtered vector bundle $E\supset L\supset 0$. It suffices to show that the tangent space to the isomorphism class of $E$ at the point $\langle\phi\rangle$ coincides with the image of $\pi_\phi\colon T^*_\phi P\to T_\phi P$. Indeed, it would imply that locally isomorphism classes of $E$ are symplectic leaves of $\pi$, and then symplectic leaves are connected components of isomorphism classes of $E$.
        
        From the filtration $E\supset L\supset 0$ on $E$ we get the spectral sequence \eqref{FO_spectral_sequence_nonreduced} that computes $\Def(E) = \Ext^1(E, E)$. From this point of view deformations of $E$ are represented by $2\times 2$-matrices, and deformations of $E\supset L\supset 0$ correspond to strictly upper triangular matrices. Hence, the term $E_2^{1, 0}$ of the spectral sequence (it corresponds to strictly upper triangular matrices) can be identified with $\Def(E\supset L\supset 0) = T_\phi P = \Ext^1(V, \O_C)/\langle\phi\rangle$, and the natural map
        $$\Def(E\supset L\supset 0)\to \Def(E)$$
        can be identified with the map
        $$E_2^{1, 0} = \Ext^1(V, \O_C)/\langle\phi\rangle\to \Ext^1(E, E)$$
        given by convergence of the spectral sequence. Hence, the latter map sends a tangent vector $v\in \Ext^1(V, \O_C)/\langle\phi\rangle$ to the induced deformation of $E$. Since the kernel of this map coincides with the image of the differential $d_2$ on the second page of the spectral sequence and $\pi_\phi$ is by definition equals to $d_2$, it follows that the image of $\pi_\phi$ consists of the tangent vectors $v\in T_\phi P$ that induce the trivial deformation of $E$, and this is exactly the tangent space to the isomorphism class of $E$ in $P$.
    \end{proof}
    
    \section{Conormal Lie algebras}\label{section_conormal_algebras}
    
    \subsection{The intrinsic derivative and conormal Lie algebras}\label{subsection_intrinsic_derivative}
    
    Following \cite{AlanWeinstein}, we use the definition of conormal Lie algebras based on the intrinsic derivative. Let us first provide a general definition of the intrinsic derivative.
    
    Let $X$ be a smooth algebraic variety and $\pi\colon E\to F$ be a morphism of vector bundles on $X$. Fix a point $x\in X$ and a tangent vector $v\in T_x X$. We will define the intrinsic derivative
    $$\partial_v\pi\colon \Ker(\pi_x)\to \Coker(\pi_x),$$
    where $\pi_x\colon E_x\to F_x$ is the map of fibers induced by $\pi$. One can find a differential-geometric definition of the intrinsic derivative in \cite{Arnold_singularities, GolubitskyGuillemin}. The intrinsic derivative may be also defined using the Atyiah class (see \cite{Markarian_AtiyahClass}) as follows. Given a morphism $\pi$ as above, consider the two-term complex $P^\bullet = (E \stackrel{\pi}{\to} F)$. Its Atyiah class is an element of hypercohomology $\Ext^1(P^\bullet, P^\bullet\otimes \Omega^1_X)$. The intrinsic derivative is given by the endomorphism of the cohomology of $P^\bullet$ induced by this class. The reader may
    try to express our subsequent manipulation with intrinsic derivative in terms of the formalism of the Atiyah class.
    
    Nevertheless, we will use another definition. Let $t_v\colon D\to X$ be the morphism corresponding to the tangent vector $v\in T_xX$. Restriction of $\pi$ along $t_v$ gives a homomorphism of free $k[\eps]$-modules
    $$\tilde \pi_x\colon \tilde E_x\to \tilde F_x,$$
    where $\tilde E_x = t_v^*E$, $\tilde F_x = t_v^*F$, and $\tilde \pi_x = t_v^*\pi$. Note that $\tilde \pi_x$ is a deformation of $\pi_x\colon E_x\to F_x$, i.e. 
    $$\tilde \pi_x\otimes_{k[\eps]} k = \pi_x.$$
    
    The definition is based on the intrinsic derivative functor $\pmb{D}$ from the derived category $\D(k[\eps])$ to the category of graded $k$-vector spaces equipped with an endomorphism of degree 1. The functor is defined by the formula
    \begin{equation}\label{intrinsic_functor}
        \pmb{D} = p'_*(-\Lotimes_{k[\eps]} \xi),
    \end{equation}
    where $\xi\in \Hom_{\D(k[\eps])}(k, k[1])$ corresponds to the extension
    $$\xymatrix{
        0\ar[r]&k\ar[r]^(0.44)\eps &k[\eps]\ar[r]&k\ar[r]&0,
    }$$
    and the functor $p'_*\colon \D(k[\eps])\to \D(k)$ is induced by the morphism $p'\colon D\to pt$. More accurately, given an object $C^\bullet$ of the category $\D(k[\eps])$, i.e a complex of $k[\eps]$-modules, we take its tensor product with the morphism $\xi\colon k\to k[1]$ and then apply the functor $p'_*$ to get a map
    $$p'_*(C^\bullet \Lotimes_{k[\eps]} k)\to p'_*(C^\bullet \Lotimes_{k[\eps]} k)[1]$$
    in the category $\D(k)$, which can be identified with the category of graded $k$-vector spaces by the cohomology functor. So we obtain a graded vector space $p'_*(C^\bullet\Lotimes_{k[\eps]} k)$ equipped with a degree 1 endomorphism. 
    
    To define the intrinsic derivative of $\pi$ consider the two term complex 
    $$C^\bullet = (C^0\to C^1) = (\tilde E_x\xrightarrow{\tilde \pi_x} \tilde F_x)$$
    and apply the intrinsic derivative functor $\pmb{D}$ to it. Since $\tilde E_x$ and $\tilde F_x$ are flat,
    $$p'_*(C^\bullet \Lotimes_{k[\eps]} k) = p'_*(C^\bullet\otimes_{k[\eps]} k) = H^\bullet(C^\bullet\otimes_{k[\eps]} k) = H^\bullet(E_x\xrightarrow{\pi_x} F_x),$$
    and we get a degree 1 endomorphism of this graded vector space. Thus, the intrinsic derivative functor gives a map
    $$\partial_v\pi\colon \Ker(\pi_x)\to \Coker(\pi_x).$$
    Let us summarize the definition of the intrinsic derivative.
    \begin{Definition}
        Let $X$ be a smooth algebraic variety over $k$, $\pi\colon E\to F$ be a morphism of vector bundles on $X$, and $v\in T_xX$ be a tangent vector to $X$ at a point $x\in X$. The intrinsic derivative 
        $$\partial_v\pi\colon \Ker(\pi_x)\to \Coker(\pi_x)$$
        is defined as follows.
        \begin{enumerate}
            \item Take the morphism $t_v\colon D\to X$ corresponding to $v\in T_x X$. Consider $k[\eps]$-modules $\tilde E_x = t_v^*E$, $\tilde F_x = t_v^*F$ and the map $\tilde \pi_x = t_v^*\pi$ between them.
            \item Apply the intrinsic derivative functor $\pmb{D} = p'_*(-\Lotimes \xi)$ to the complex $(\tilde E_x\xrightarrow{\tilde \pi_x} \tilde F_x)$ to get a degree 1 endomorphism of $H^\bullet(E_x\xrightarrow{\pi_x} F_x)$, i.e a map
            $$\partial_v\pi\colon \Ker(\pi_x)\to \Coker(\pi_x).$$
        \end{enumerate}
    \end{Definition}
    The following proposition gives a description of the intrinsic derivative which is well-known in differential geometry.
    \begin{Proposition}\label{Proposition_derivative}
        Let $X$ be a smooth algebraic variety, $\pi\colon E\to F$ a morphism of vector bundles on $X$, and $v\in T_xX$ a tangent vector to $X$ at a point $x\in X$. Choose trivializations of $E$ and $F$ in a neighborhood $U$ of $x$ and write $\pi$ as a matrix $\Pi$ of functions on $U$. Then the intrinsic derivative $\partial_v\pi$ is the composition
        $$
        \xymatrix@C10.0ex{
            \Ker(\pi_x)\ar[d]\ar[r]^{\partial_v\pi}& \Coker(\pi_x)\\
            E_x\ar[r]^{v(\Pi)}& F_x,\ar[u]
        }
        $$
        where $v(\Pi)$ is the element-wise derivative of the matrix $\Pi$ along $v$, and left and right vertical arrows are the natural embedding of the kernel and the natural projection onto the cokernel, respectively.
    \end{Proposition}
    \begin{proof}
        Restricting $\pi$ along the morphism $D\to X$ corresponding to the tangent vector $v$, we get a map of free $k[\eps]$-modules given by the matrix $\Pi_x + \eps \cdot v(\Pi)$. The rest of the proof is a straightforward calculation of the intrinsic derivative functor $\pmb{D}$ applied to this map, so we omit it.
    \end{proof}
    
    Now let us, following \cite{AlanWeinstein}, formulate a definition of conormal Lie algebras of a Poisson structures. Let $\pi\colon T^*X\to TX$ be a Poisson structure on a smooth algebraic variety $X$ and $x\in X$ be a point.
    Since $\pi_x$ is skew-symmetric, the natural pairing 
    $$\langle-, -\rangle\colon T_xX\otimes T^*_xX\to k$$ 
    induces a natural isomorphism
    $$\Coker(\pi_x) = \Ker(\pi_x)^*.$$
    
    The conormal Lie algebra is the vector space $\Ker(\pi_x)$ equipped with a Lie bracket that can be constructed as follows. For any tangent vector $v\in T_xX$ consider the intrinsic derivative
    $$\partial_v\pi\colon \Ker(\pi_x)\to \Coker(\pi_x).$$
    It is clear from Corollary \ref{Proposition_derivative} that $\partial_v\pi$ is linear in $v$, so we get a map
    \begin{align*}
        \partial\pi\colon T_xX\otimes \Ker(\pi_x)&\to \Coker(\pi_x),\\
        v\otimes \alpha&\mapsto \partial_v\pi(\alpha).
    \end{align*}
    Moreover, one can check that if $v\in \Im(\pi_x)$ then the map $\partial_v\pi$ is zero, so in fact we have
    $$\partial\pi\colon \Coker(\pi_x)\otimes \Ker(\pi_x)\to \Coker(\pi_x).$$
    Using the isomorphism $\Coker(\pi_x) = \Ker(\pi_x)^*$ and dualizing, we get a map
    \begin{equation}\label{conormal_bracket}
        \Ker(\pi_x)\otimes \Ker(\pi_x)\to \Ker(\pi_x).
    \end{equation}
    This map is clearly skew-symmetric and, in fact, it is a Lie bracket on $\Ker(\pi_x)$.
    \begin{Definition}
        Let $x\in X$ be a point on a smooth Poisson variety $(X, \pi)$. The conormal Lie algebra at $x$ is the vector space $\Ker(\pi_x)$ equipped with the constructed Lie bracket \eqref{conormal_bracket}.
    \end{Definition}
    One of the reasons why conormal Lie algebras are important is the following observation. If a point $x\in X$ is contained in a symplectic leaf $Z\subset X$ then $\Ker(\pi_x)$ can be identified with the conormal space $(N^*_Z X)_x$ of $Z$, and then one can view the conormal Lie bracket on $\Ker(\pi_x)$ as a linearization of the Poisson structure $\pi$ in a neighbourhood of $Z$.
    
    \subsection{Conormal Lie algebras of Feigin--Odesskii Poisson structures}\label{subsection_conormal_FO}
    
    Let $V$ be a stable vector bundle of degree $n>0$ on an elliptic curve $C$ and $P = \mb{P}\Ext^1(V, \O_C)$ be the moduli space of filtered vector bundles $E\supset L\supset 0$ with fixed associated quotients $E/L\simeq V$, $L\simeq \O_C$.
    \begin{Theorem}\label{thm_conormal_algebras}
        Let $\langle\phi\rangle$ be a point of $P$ corresponding to a filtered vector bundle $E\supset L\supset 0$. Then the conormal Lie algebra of the Feigin--Odesskii Poisson structure $\pi$ at the point $\langle \phi\rangle$ is isomorphic to the Lie algebra $\End(E)_0$ of traceless endomorphisms of $E$. 
    \end{Theorem}
    The rest of the section consists of the proof of this theorem. We need to compute the intrinsic derivative of $\pi\colon T^*P\to TP$. Let $\langle \phi\rangle\in P$ be a point corresponding to a filtered vector bundle $E\supset L\supset 0$ on $C$ and let $v\in T_\phi P$ be a tangent vector corresponding to a deformation $\tilde E\supset \tilde L\supset 0$. Recall from Section \ref{subsection_deformations} that we have two Cartesian squares
    $$
    \xymatrix{
        C\times D\ar[r]^(0.55)q\ar[d]_p & D\ar[d]^{p'} && C\times D\ar[r]^(0.55){q} & D\\
        C\ar[r]^{q'} & pt, && C\ar[r]^{q'}\ar[u]^i & pt\ar[u]_{i'},
    }
    $$
    and by Definition \ref{definition_filtered_deformation},
    $$i^*\tilde E = E,\quad i^*\tilde L = L,$$
    $$\tilde E/\tilde L\simeq p^*V,\quad \tilde L\simeq p^*\O_C.$$
    
    In order to compute the intrinsic derivative $\partial_v\pi$ we need to restrict $\pi$ along the morphism $t_v\colon D\to P$. The functorial description of $\pi$ shows that $\tilde \pi_\phi = t_v^*\pi$ is the differential $d_2$ on the second page of the spectral sequence of $k[\eps]$-modules given by a filtration on $\sEnd(\tilde E, \tilde E)_0$. This spectral sequence computes $\Ext^\bullet(\tilde E, \tilde E)_0 = R^\bullet q_*\sEnd(\tilde E)_0$, and it has the following first page.
    \begin{equation}\label{deformed_spectral_sequence}
        \xymatrix{
            \ar@{-}@<4ex>[ddd]  &      \hphantom{-----}       &    \hphantom{-----}    &  \hphantom{-----}   &         \\
            1  &  \Hom(\O_C, V)\otimes_k k[\eps]\ar[r]  & k[\eps]&   &\\
            0  &       & k[\eps]\ar[r] & \Ext^1(V, \O_C)\otimes_k k[\eps]  &\\
            \ar@{-}@<4ex>[rrrr]  & -1             & 0      & 1&
        }
    \end{equation}
    If we take the tensor product of this spectral sequence with $k$, we will get the spectral sequence \eqref{FO_definition_spectral_sequence} with $\pi_\phi$ on the second page. Hence, differentials in the first and zeroth rows of \eqref{deformed_spectral_sequence} are surjective and injective, respectively. The theorem will follow from two lemmas.
    \begin{Lemma}\label{lemma_second_page}
        The second page of the described spectral sequence, i.e. the two term complex with the differential $d_2$, is quasi-isomorphic to $R q_*\sEnd(\tilde E)_0$.
    \end{Lemma}
    \begin{proof}
        Let us for brevity write the two term complex on the second page as $C^0\xrightarrow{\tilde \pi_\phi}C^1$: the second page looks as follows.
        $$
        
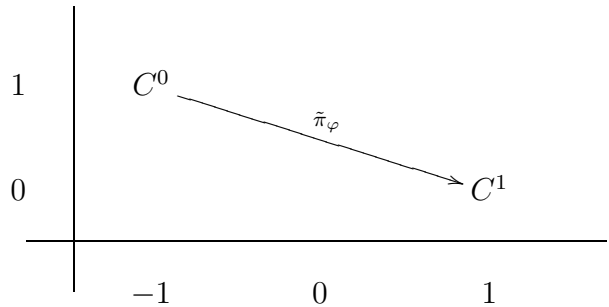
\xymatrix{
            \ar@{-}@<4ex>[ddd]  &      \hphantom{---}       &    \hphantom{---}    &  \hphantom{---}   &         \\
            1  &  C^0\ar[rrd]^{\tilde \pi_\phi}  & &   &\\
            0  &       &  & C^1  &\\
            \ar@{-}@<4ex>[rrrr]  & -1             & 0      & 1&
        }
        $$
        Note that the spectral sequence actually comes from a filtration on the complex $Rq_*\sEnd(\tilde E)_0$. If we apply the d\'ecalage construction \cite[Exercise 5.4.3]{Weibel} two times, we will get another filtration on $Rq_*\sEnd(\tilde E)_0$ such that the induced spectral sequence has the following zeroth page.
        $$
        
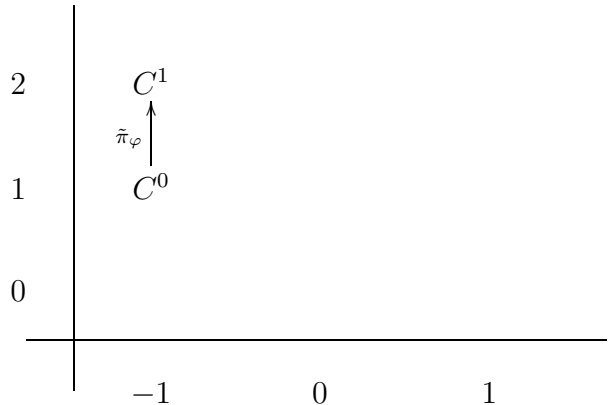
\xymatrix{
            \ar@{-}@<4ex>[dddd]  &      \hphantom{---}       &    \hphantom{---}    &  \hphantom{---}   &         \\
            2  &    C^1   &  &   &\\
            1  &  C^0\ar[u]^{\tilde \pi_\phi}  & &   &\\
            0  &       &  &   &\\
            \ar@{-}@<4ex>[rrrr]  & -1             & 0      & 1&
        }
        $$
        Thus, we get a finite filtration on $Rq_*\sEnd(\tilde E)_0$ such that all but one of the associated quotients are trivial as objects of $\D(k[\eps])$, and the only non-trivial quotient is quasi-isomorphic to $C^0\xrightarrow{\tilde \pi_\phi} C^1$. One can easily deduce from this (e.g. by induction) that then the whole complex $Rq_*\sEnd(\tilde E)_0$ is quasi-isomorphic to $C^0\xrightarrow{\tilde \pi_\phi} C^1$.
    \end{proof}
    
    \begin{Lemma}\label{lemma_intrinsic_to_deformation}
        Let $\tilde E$ be a deformation of a locally free sheaf $E$ on $C$ corresponding to an element $T\in \Ext^1(E, E)$. Then $\pmb{D}(Rq_*\tilde E)$ is the graded vector space $H^\bullet(C, E)$ with the degree 1 endomorphism
        $$H^\bullet(C, E)\xrightarrow{T} H^{\bullet + 1}(C, E).$$
    \end{Lemma}
    \begin{proof}
        Recall from Section \ref{subsection_deformations} that
        \begin{equation}\label{deformation_on_C}
            T = p_*(\tilde E\otimes q^*\xi)\colon E\to E[1] \text{ in } \D(C).
        \end{equation}
        Moreover, since by Lemma \ref{lemma_deformation}, $\tilde E$ is a locally free sheaf on $C\times D$, we can rewrite \eqref{deformation_on_C} using derived tensor product:
        $$T = p_*(\tilde E\Lotimes q^*\xi)\colon E\to E[1] \text{ in } \D(C).$$
        Then by the projection formula and flat base change property,
        $$\pmb{D}(Rq_*\tilde E) = p'_*((Rq_*\tilde E)\Lotimes \xi) = p'_*(Rq_*(\tilde E\Lotimes q^*\xi)) = Rq'_*(p_*(\tilde E\Lotimes q^*\xi)) = Rq'_*(T),$$
        and this is exactly the endomorphism of $H^\bullet(C, E)$ induced by $T$.
    \end{proof}
    
    Now we are ready to finish the proof of Theorem \ref{thm_conormal_algebras}. In order to compute the intrinsic derivative of $\pi$ we need to apply the intrinsic derivative functor to the second page of the spectral sequence \eqref{deformed_spectral_sequence}. Since by Lemma \ref{lemma_second_page} the second page is quasi-isomorphic to $Rq_*\sEnd(\tilde E)_0$ and the functor $\pmb{D}$ is well-defined on the derived category, we will apply $\pmb{D}$ to $Rq_*\sEnd(\tilde E)_0$. Taking the traceless part of the third statement of Lemma \ref{lemma_deformation_tensor}, we obtain that $\sEnd(\tilde E)_0$ is the deformation of $\sEnd(E)_0$ corresponding to the element $[-, T]\in \Ext^1(\sEnd(E)_0, \sEnd(E)_0)$. Finally, by Lemma \ref{lemma_intrinsic_to_deformation}, $\pmb{D}(Rq_*\sEnd(\tilde E)_0)$ gives the map
    $$H^0(C, \sEnd(E)_0)\xrightarrow{[-, T]}H^1(C, \sEnd(E)_0).$$
    Thus, the intrinsic derivative of the Feigin--Odesskii Poisson structure at the point corresponding to a filtered vector bundle $E\supset L\supset 0$ along the tangent vector corresponding to a deformation $\tilde E\supset \tilde L\supset 0$ is the map
    $$\End(E)_0\xrightarrow{[-, T]}\Ext^1(E, E)_0,$$
    where $T\in \Ext^1(E, E)$ corresponds to the deformation $\tilde E$ of $E$. Then the induced Lie bracket on $\Ker(\pi_\phi) = \End(E)_0$ coincides with the natural commutator
    $$\End(E)_0\otimes \End(E)_0\xrightarrow{[-, -]} \End(E)_0.$$
    The proof of Theorem \ref{thm_conormal_algebras} is complete.
    
    \printbibliography
\end{document}